\documentclass[12pt,a4paper,reqno]{amsart}
\usepackage{latexsym}
\usepackage{amssymb}
\usepackage{enumitem}

\def \zgran{\displaystyle}

\def \za{\alpha}
\def \zb{\beta}
\def \zg{\gamma}
\def \zd{\delta}

\def \zh{\theta}

\def \zl{\lambda}
\def \zm{\mu}

\def \zp{\pi}

\def \zr{\rho}

\def \zf{\varphi}

\def \zsu{\sum}

\def \zpor{\times}

\def \zmei{\leq}
\def \zmai{\geq}
\def \zco{\subset}

\def \zpe{\in}

\def \zpar{\partial}

\def \zfl{\rightarrow}

\def \zbv{\mid}

\def \z/{\over}

\def \zdp{\colon}
\def \zps{\dots}

\newcommand {\ccB}{\mathcal B}

\newcommand {\nnR}{\mathbb R}
\newcommand {\nnT}{\mathbb T}
\newcommand {\nnZ}{\mathbb Z}

\newtheorem*{theorem*}{Theorem}
\newtheorem{lemma}{Lemma}

\newtheorem{proposition}{Proposition}

\newtheorem{remark}{Remark}

\title{ On the rank of a product of manifolds}

\author{Francisco-Javier~Turiel}
\address[F.J.~Turiel]{
Geometr{\'\i}a y Topolog{\'\i}a,
Facultad de Ciencias, Campus de Teatinos, s/n,
29071-M{\'a}laga, Spain}
\email {turiel@uma.es}

\author{Arthur G.~Wasserman}
\address[A. G.~Wasserman]{
University of Michigan, Ann Arbor,
MI 48109-1003, USA}
\email{awass@umich.edu}

\thanks{The first author is partially supported by
MEC-FEDER grant MTM2013-41768-P, and JA grants FQM-213}

\begin{document}

\begin{abstract}
This note gives an example of closed smooth manifolds $M$ and $N$ for which 
the rank of $M\times N$ is strictly greater than $rank M + rank N$.
\end{abstract}

\maketitle

Milnor defined the {\it rank} of a smooth manifold $M$ as  the maximal number of   commuting vector fields 
on $M$ that are linearly independent at each point. 
 
One of the questions  raised by Milnor at the Seattle Topology Conference 
of 1963,  and echoed by Novikov \cite{No}, was 
 $$ is\  rank(M\zpor N)=rank(M)+rank(N)$$
whenever $M$ and $N$ are smooth closed manifolds? 

{\it In this note we give a negative answer to this question.}

We need a simple result about mapping tori.

Let $f\zdp X\zfl X$ be a diffeomorphism of a manifold X and let    
$$M(f)={\frac {I\zpor X} {(0,x)_{\widetilde~} (1,f(x))}} $$
be the mapping torus of f where $I=[0,1]$.  
  
Equivalently,   $M(f)=\zgran {\frac {\nnR\zpor X} {\nnZ}}$ where the action of 
$\nnZ$ on $\nnR\zpor X$ is given by $\za(k)(t,x)=(t+k,f^k (x))$.
$M(f)$ is a fibre bundle over $S^1$  with fiber $X$. We note that 
$\zp_1 (M(f))=\zp_1 (X)\ast_{f}\nnZ$ where $\ast$ denotes the semi-direct product and
$f_\ast\zdp\zp_1 (X)\zfl\zp_1 (X)$.

\begin{proposition} \label{P1}
Consider two periodic diffeomorphisms $f\zdp X\zfl X$ and $g\zdp Y\zfl Y$ with periods
$m$ and $n$ respectively.
Assume  $m$ and $n$ are  relatively prime, i.e., there are integers $c,d$ such that $mc+nd=1$.

Then $M(f) \zpor  M(g)$ is diffeomorphic to $ M(h)$ where 
$h\zdp S^1 \zpor X\zpor Y\zfl S^1 \zpor X\zpor Y$ is defined by $h(\zh,x ,y)=(\zh,f^{-d}(x ),g^c (y ))$.
Moreover $h^{m-n}=(id,f,g)$.
\end{proposition}

\begin{proof}
$M(f) \zpor  M(g)$ can be identified with the quotient of $\nnR^2 \zpor X\zpor Y$ under the action
of $\nnZ^2$ given by $\zb(z)(u,x,y)=(u+z, f^{z_1}(x), g^{z_2}(y ))$, where
$z=(z_1 ,z_2 )\zpe\nnZ^2$, $u=(u_1 ,u_2 )\zpe\nnR^2$ and $(x ,y)\zpe X\zpor Y$.

Set $\zl=(m,n)$ and $\zm=(-d,c)$. 
Since $mc+nd=1$, $\ccB=\{\zl,\zm\}$ is at the same time a basis of $\nnZ^2$ as a
$\nnZ$-module and a basis of $\nnR^2$ as a vector space. On the other hand
$$\zb(\zl)(u,x,y)=(u+\zl,x,y)\quad  \text{and}\quad  \zb(\zm)(u,x,y)=(u+\zm,f^{-d} (x ),g^c (y )).$$

Therefore the action $\zb$ referred to the new basis $\ccB$ of $\nnZ^2$ and $\nnR^2$ is
written now:
$$\zb(k,r)(a,b,x ,y)=(a+k,b+r,{\zf}^r (x ),{\zg}^r (y ))$$
where $\zf=f^{-d}$ and $\zg=g^c$. 

As the action of the first factor of $\nnZ^2$ on $X\zpor Y$ is trivial, identifying 
$S^1$ with $\zgran{\frac {\nnR} {\nnZ}}$ shows that $M(f) \zpor  M(g)$ is 
diffeomorphic to $M(h)$. 

Finally from $(-n)(-d)=1-cm$ and $cm=1-dn$ follows that $h^{m-n}=(id,f,g)$. 
\end{proof}

On the other hand:

\begin{lemma} \label{L1}
Let $f\zdp N\zfl N$ be a diffeomorphism and let $X_1 ,\dots,X_k$ be a family of commuting
vector fields on $N$ that are linearly independent everywhere. Assume 
$f_\ast X_i =\zsu_{j=1}^k a_{ij}X_j$, $i=1,\zps,k$, where the matrix 
$(a_{ij})\zpe GL(k,\nnR)$. Then $rank(M(f))\zmai k$.
\end{lemma}

\noindent{\bf Proof.} 
It suffices to construct $k$ commuting vector fields ${\widetilde X}_1 ,\zps,{\widetilde X}_k$ on 
$I\zpor N$ that are linearly independent at each point and such that every ${\widetilde X}_i (t,x)$ equals
$X_i (x)$ if $t$ is close to zero and $f_\ast X_i (x)$ when $t$ is close to 1 ($X_1 ,\dots,X_k$
are considered vector fields on $I\zpor N$ in the obvious way). 

If $\zbv a_{ij}\zbv>0$ consider an interval $[a,b]\zco (0,1)$ and a (differentiable) map
$(\zf_{ij} )\zdp I\zfl GL(k,\nnR)$ such that $\zf_{ij}([0,a])=\zd_{ij}$ and 
$\zf_{ij}([b,1])=a_{ij}$, and set ${\widetilde X}_i (t,x) =\zsu_{j=1}^k\zf_{ij}(t)X_j (x)$.

When $\zbv a_{ij}\zbv<0$ first take an interval $[c,d]\zco (0,1/2)$ and a function
$\zr\zdp [0,1/2]\zfl \nnR$ such that $\zr([0,c])=1$, $\zr([d,1/2])=-1$, and on
$[0,1/2]\zpor N$ set ${\widetilde X}_1 (t,x)=\zr(t)X_1 (x)+(1-\zr^2 (t)){\frac{\zpar} {\zpar t}}$
and ${\widetilde X}_i (t,x)=X_i (x)$, $i=2,\zps,k$.
 
The matrix of coordinates  of $f_\ast X_1 ,\zps,f_\ast X_k$ with respect to the basis 
$\{-X_1 ,X_2 ,\zps,X_k \}$  has positive determinant, so by doing as before we can extend
${\widetilde X}_1 ,\zps,{\widetilde X}_k$ to $[1/2,1]\zpor N$ by means of an interval
$[a,b]\zco (1/2,1)$ and a suitable map $(\zf_{ij} )\zdp [1/2,1]\zfl GL(k,\nnR)$. 
$\quad\square$

%%%
$\,$

Proposition \ref{P1} and Lemma \ref{L1} quickly yield a counterexample.

$\,$

%%%

Assume $X$ is a torus $\zgran\nnT^k ={\frac {\nnR^k} {\nnZ^k}}$ and $f$ is the map
induced by a nontrivial element of $GL(k,\nnZ)$. Then by the above lemma 
applied to ${\frac {\zpar} {\zpar\zh_j}}$, $j=1,\dots,k$, $rank(M(f))\zmai k$.
But $M(f)$ has non-abelian fundamental group, so it is not a torus and $rank(M(f))=k$. 
(If M is a closed connected n-manifold of rank n then M is diffeomorphic to the n-torus.)  

For the same reason if $Y=\nnT^r$ and $g$ is induced by a nontrivial element of
$GL(r,\nnZ)$ then $rank(M(g))=r$. 

If $f$ and $g$ are periodic with relatively
prime periods $m$ and $n$ respectively  then by Proposition \ref{P1} $M(f) \zpor  M(g)$ =$M(h)$ where
$h\zdp\nnT^{k+r+1}\zfl\nnT^{k+r+1}$ is induced by a nontrivial element of $GL(k+r+1,\nnZ)$. 
Moreover $rank(M(h))=k+r+1$ . Therefore:
$$rank(M(f)\zpor M(g))>rank(M(f))+rank(M(g)).$$

For instance, set $k=r=2$ and consider $f,g$ induced by the elements in 
$SL(2,\nnZ)\zco GL(2,\nnZ)$
\[\begin{pmatrix} -1 & 0 \\ 0&-1 \\ \end{pmatrix}
\quad\text{and}\quad
\begin{pmatrix} 0&1\\ -1&-1\\ \end{pmatrix}\]
respectively, so $M(f)$ and $M(g)$ are orientable. Then  the period of $f$ is 2 and that of $g$
equals 3 .

An even simpler but non-orientable counterexample can be constructed as follows.
Take $r$ and $g$ as before, $k=1$ and $f$ induced by $(-1)$.
Then $M(f)$ is the Klein bottle which has rank 1 and M(g) has rank 2; however, 
$M(f) \zpor  M(g)$  is diffeomorphic to $ M(h)$ and hence has rank 4.

$\,$

\begin{remark} \label{BB}
{\rm The {\it file} of a manifold $M$ was defined by Rosenberg \cite{Ro} to be the largest integer 
$k$ such that $\nnR^k$ acts locally free on $M$. When $M$ is closed $file(M)$ equals $rank(M)$ but 
$file(\nnR\zpor S^2 )=1$, \cite{Ro}, while $rank(\nnR\zpor S^2 )=3$.

The analog of Milnor's question for the file of a product of noncompact manifolds also fails. Indeed, 
let $\nnR_{e}^4$ be any exotic $\nnR^4$. Then $file(\nnR_{e}^4 )\zmei 3$
otherwise  $\nnR_{e}^4 =\nnR^4$. But $\nnR_{e}^4 \zpor\nnR=\nnR^5$
because there in no exotic $\nnR^5$, so
$file(\nnR_{e}^4 \zpor\nnR)=5>file(\nnR_{e}^4 )+file(\nnR)$.

Orientable closed connected $n$-manifolds of rank $n-1$ are completely 
described in \cite{RRW,CR,Ti}.}
\end{remark}

%%%%%%%%%%%%%%%%%%%%%%%%%%%%%%%%%%
% Bibliography
%%%%%%%%%%%%%%%%%%%%%%%%%%%%%%%%%%


\begin{thebibliography}{99}


\bibitem{CR} G. Chatelet and H. Rosenberg, {\it  Manifolds which admit $\nnR^n$ actions}, 
Inst. Hautes \'Etudes Sci. Publ. Math. {\bf 43} (1974), 245--260.

\bibitem{No}  S. P. Novikov, {\it The Topology Summer Institute (Seattle, USA, 1963)}, 
Russian Math. Surveys {\bf 20} (1965), 145--167.
http://www.mi.ras.ru/$_{\widetilde~}$snovikov/16.pdf.

\bibitem{Ro}  H. Rosenberg,  
{\it Singularities of $\nnR^2$ actions}, 
Topology {\bf 7} (1968), 143--145.

\bibitem{RRW}  H. Rosenberg, R. Roussarie and  D. Weil, 
{\it A classification of closed oriented 3-manifold of rank two}, 
Ann. of Math. {\bf 91} (1970), 449--464.

\bibitem{Ti} D. Tischler, {\it Manifolds $M^n$ of rank $n-1$},
 Proc. Amer. Math. Soc.  {\bf 94} (1985),158--160.

\end{thebibliography}
\end{document}